\documentclass[12pt]{amsart}
\usepackage{a4, amsthm} 
\usepackage[all]{xypic}

\theoremstyle{plain}
\newtheorem{theorem}{Theorem}[section]
\newtheorem{lem}[theorem]{Lemma}
\newtheorem{pro}[theorem]{Proposition}
\newtheorem{cor}[theorem]{Corollary}

\theoremstyle{definition}

\newtheorem{rem}[theorem]{Remark} 

\def\Z{{\mathbb{Z}}}

\def\C{{\mathbb{C}}}
\def\R{{\mathbb{R}}}

\def\H{{\mathbb{H}}}
\def\N{{\mathbb{N}}}
\def\P{{\mathbb{P}}}

\def\Ker{{\rm Ker}}

\def\nsmallskip{\smallskip\noindent}
\def\bbigskip{\bigskip\bigskip}
\def\nbigskip{\bigskip\noindent}
\def\nbbigskip{\bbigskip\noindent}

\def\nmedskip{\medskip\noindent}
\def\buildunder#1#2{\mathrel{\mathop{\kern0pt #2} \limits_{#1}}}
\def\buildover#1#2{\buildrel#1\over#2}

\def\Aut{\hbox{Aut}}
\def\qq{/\kern-.185em /}

\begin{document}

\title[taut Stein surfaces ]{A classification of taut, Stein surfaces with a proper $\,\R$-action}

\author[Iannuzzi]{A. Iannuzzi}
\author[Trapani]{S. Trapani}
\address{Andrea Iannuzzi and Stefano Trapani: Dip.\ di Matematica,
Universit\`a di Roma  ``Tor Vergata", Via della Ricerca Scientifica,
I-00133 Roma, Italy.} 
\email{iannuzzi@mat.uniroma2.it, trapani@mat.uniroma2.it}



\begin{abstract}
We present a classification of 2-dimensional, taut, Stein manifolds with
a proper $\,\R$-action. For such manifolds
the globalization with respect to the induced local $\,\C$-action 
turns out to be  Stein. As an application we determine all 2-dimensional
taut, non-complete, Hartogs domains over a Riemann surface.  
\end{abstract}


\maketitle

\section{Introduction}

\smallskip
Let the group $\,(\R, +)\,$ act on a complex manifold $\,X\,$ by biholomorphism. Then, by integrating
the associated vector field one obtains a local action of $\,(\C,+)$.
For taut, Stein manifolds,  the universal globalization with respect to such a local action is Hausdorff (\cite{Ian}). That is, there
exists a complex $\,\C$-manifold $\,X^*\,$ containing $\,X\,$ as an $\,\R$-invariant 
domain such that every $\,\R$-equivariant holomorphic map from $\,X\,$ onto a complex $\,\C$-manifold
extends $\,\C$-equivariantly on $\,X^*$. Recently C. Miebach and K. Oeljeklaus have shown  that if $\,X\,$ is 2-dimensional and the $\,\R$-action is proper, then
the $\,\C$-action on $\,X^*\,$ is also proper,
implying
that the globalization $\,X^*\,$ can be regarded as a holomorphic principal $\,\C$-bundle
over the Riemann surface $\,S:=X^*/\C\,$  (\cite{MiOe}). 

Our main goal here is to present a classification of all such $\,X$, up to
$\,\R$-equivariant biholomorphism.
We first exploit the above bundle structure in order to give a more precise description of $\,X^*$.
In the case when  $\,S\,$ is non compact,  $\,X^*\,$ is $\,\C$-equivariantly biholomorphic
to $\,\C \times S$, where $\,\C\,$ acts by translations on the first factor.
 If the base $\,S\,$ is compact, then it is hyperbolic and $\,X^*\,$ turns out to be
 $\,\C$-equivariantly biholomorphic
to a certain twisted bundle $\,\C \times \Delta/\Gamma$, where $\,\Delta\,$  is the unit disk in  $\,\C\,$
and $\,\Gamma\,$ is the group of deck transformations of the universal covering
$\,\Delta \to S$. Then, by using a result of T. Ueda (\cite{Ued}) as the main ingredient,
we prove the following 

 \nbigskip
{\bf Theorem.} {\it  Let $\,X\,$ be a 2-dimensional, taut, Stein manifold with a proper
$\,\R$-action. Then its universal globalization $\,X^*\,$ is Stein.}

\bigskip
Note that in the more general context of Stein $\,\R$-manifolds it is an open problem to determine whether $\,X^*\,$ is always Stein or at least Hausdorff (cf.$\,$\cite{HeIa}, \cite{CIT}, 
\cite{IST}).
Once $\,X^*\,$ is understood, we look at the realization of $\,X\,$ as an $\,\R$-invariant
domain of $\,X^*\,$ and the following question turns out to be crucial.
Given  an upper semicontinuous function
$\,a: \C \to \{-\infty\} \cup \R$, consider the  ($\R$-invariant) subdomain of  $\,\C^2\,$ 
defined by 
$$\,\Omega_a:= \{\,(z,w) \in \C^2 \ : \ a(w)<{\rm Im}\,z \}\,.$$
Under which conditions on $\,a\,$ is $\,\Omega_a\,$  taut ?
Since taut domains in $\,\C^n\,$ are Stein, the function $\,a\,$ is necessarily 
subharmonic.
Moreover $\,\Omega_a\,$ cannot contain 
 complex lines, therefore $\,a(w) > -\infty\,$ for all $\,w \in \C$.

Partial answers to this problem can be found, e.g.,  in \cite{Yu} and  \cite{Gau}.
Here the following necessary and sufficient condition 
is obtained by using tools of potential theory
(Thm.$\,$\ref{TAUTT}).
 
 \nbigskip
{\bf Theorem.} {\it  The domain $\,\Omega_a\,$ is taut if and only if 
 $\,a\,$ is real valued, subharmonic, non-harmonic and continuous.}

\bigskip

This result put us in the position  of showing  that 
the  2-dimensional manifolds listed below are
all taut and Stein.

\smallskip
\nmedskip
{\bf Type CH$\ $} {\it  If  $\,S\,$ is compact hyperbolic, say
$\,S=\Delta/\Gamma$, the models are certain twisted bundles
$\,H \times \Delta/\Gamma$, with $\,H\,$ a proper, $\,\R$-invariant, connected
strip in $\,\C$.

\nmedskip
{\bf Type NCH$\ $} If  $\,S\,$ is non compact hyperbolic, the models are
 $$\, \{\, (z,p) \in \C \times S : a(p) < {\rm Im}\, \,z < - b(p) \,  \}\,,$$
 where $\,a\,$ and $\, b\,$ are  subharmonic,  continuous functions on $\,S\,$
such that $\, a +b < 0\,$ and  $\, \max \{a(p),b(p)\}  >   - \infty\, $  for all $\,p \in S$.

\nmedskip
{\bf Type NCNH$\ $} If  $\,S=\C\,$ or  $\,S=\C^*$, the models are
$$\,\{\, (z,p) \in \C \times S \ :\  a(p) < {\rm Im} \,z \, \}\quad {\rm or} \quad 
\{\, (z,p) \in \C \times S \ : \ {\rm Im} \,z < -b(p) \, \}\, ,$$
with $\,a$, $\,b\,$   subharmonic, non-harmonic,  real valued,  continuous functions
 on $\,S$.}
 
 \smallskip
 \medskip
On each such manifold let $\,\R\,$ act by translations on the first factor.
Then the classification follows by proving that a  2-dimensional, taut, Stein manifold $\,X\,$
with a proper $\,\R$-action is $\,\R$-equivariantly biholomorphic
to a model as above and its type
depends on compactness and hyperbolicity of the base $\,S\,$
(cf.$\,$Thm.$\,$\ref{MAIN}). 
We recall that in the non compact,
simply connected case, a partial result is obtained in \cite{MiOe}, Theorem\,6.3.

It is worth noting  that  $\,X\,$ turns out to be 
homotopically equivalent to its base $\,S$.
As a consequence, the corresponding type 
is strongly related to the topology of $\,X$.
For instance, $\,X\,$ is of type CH if
and only if  $\,H^2(X,\Z)\not=0\,$ (cf.$\,$Sect.$\,$6).
  
We also wish  to recall   that every taut manifold is Kobayashi
hyperbolic, therefore  its automorphism group 
is a Lie group acting properly on $\,X\,$  (see \cite {Kob},$\,$Thm.$\,$5.4.2).
It follows that there exists a proper $\,\R$-action on $\,X\,$
if and only if the connected component of the identity in $\,Aut(X)\,$
is non compact (cf.$\,$\cite{Hoc}$\,$p.$\,$180,$\,$\cite{MiOe}$\,$Lemma$\,$6.3).

As an  application of the above classification 
we determine all  2-dimensional, taut, non-complete
Hartogs domains over a Riemann surface
(Prop.$\,$\ref{HARTOGS}). For a  characterization of complete Hartogs 
domains see  \cite{ThDu}, \cite{Par}.

   \nmedskip
   
  The paper is organized as follows. 
   In Section 2 we point out a characterization of taut manifolds
 and collect those results which are used in the sequel.
  
 \noindent
 In Section 3 we characterize  those domains 
of the form $\,\Omega_a\, $ which are taut (Thm. \ref{TAUTT}).

 \noindent
 In Section 4 we study models of type CH and show that
  their globalization is
 Stein. We also prove  that if the base $\,S\,$ is compact, then $\,X\,$
 is  $\,\R$-equivariantly biholomorphic to one of these models.

  \noindent
 In Section 5 the analogous results are proved for models of type NCH and NCNH.
  
  \noindent
  In Section 6 we point out that in most 
  cases the type of $\,X\,$ is determined by  the topology of $\,X\,$
  (Cor.\,\ref{TOPOL} and Rem.\,\ref{ESEMPI})
  
  \noindent
  In Section 7 we  classify 
   2-dimensional, taut, non-complete
Hartogs domains over a Riemann surface.

\nbigskip
{\bf Acknowledgments.}
We are grateful to Christian Miebach for valuable remarks and
for pointing out to us a gap in the proof of a previous version of Proposition \ref{TYPECH}.
We also wish to thank Professor Takeo Ohsawa for kindly
indicating to us the result of T. Ueda as a suitable tool for 
a complete proof of such a proposition
and Simone Diverio for his worthy comments.

 \bbigskip

 \hfill


\section{Preliminaries}
 
 \bigskip
 
By definition a complex manifold is taut if and only if every sequence of holomorphic maps  $\,f_n:\Delta \to X\,$ admits  a subsequence which is either converging
uniformly on compact subsets or compactly divergent. If $\,X\,$ is taut, then it is
hyperbolic (\cite{Kob},$\,$Thm.$\,$5.1.3).
 We first recall a result of M. Abate
and give a characterization of taut manifolds.


\bigskip
\begin{theorem}
\label{ABATE} {\rm (\cite{Aba}, Thm. 1.3)}
Let $\,X\,$ be a  complex manifold and $\,X\cup \{\infty\}\,$ its Alexandroff compactification.
Then $\,X\,$ is hyperbolic if and only if $\,Hol(\Delta, X)\,$ is relatively compact in $\,C(\Delta, X
\cup \{\infty\})\,$ with respect to the compact-open topology. 
\end{theorem}

\bigskip
Note that  since $\,X\cup \{\infty\}\,$ is metrizable,  the compact open topology
of $\,C(\Delta, X \cup \{\infty\})\,$
coincides with the topology of uniform convergence on compact subsets.


\bigskip
\begin{pro}
\label{CHAR}
For a complex manifold $\,X\,$ the following conditions 
are equivalent.

\medskip
\item{\rm (i)} $\,X\,$ is taut,

\medskip
\item{\rm (ii)} for every sequence of
holomorphic maps $\,f_n : \Delta \rightarrow X\,$ such that
$\,f_n(\zeta_0) \to x_0\,$ for some $\,\zeta_0 \in \Delta \,$
and $\,x_0 \in X$, there exists a  subsequence 
converging uniformly on compact sets of $\,\Delta$. 
\end{pro}

\begin{proof}
Condition (ii) is clearly satisfied if $\,X\,$ is taut.
Assume that (ii) holds true. We first show that $\,X\,$ is hyperbolic.
Let $\,K_X\,$ denote the Kobayashi infinitesimal pseudo metric and assume
by contradiction that there exists $\,x_0\,$ in $\,X\,$ and a non zero vector $\,v\,$
in the tangent space $\,T_{x_0}X\,$ such that  $\,K_X(v) = 0$.
Then, by definition of $\,K_X\,$ 
there exists a sequence  $\,f_n : \Delta \rightarrow X\,$ of holomorphic maps, such that $\,f_n(0) = x_0\,$ and  $\,\| {f'}_n(0)\| \to +\infty$,
 where $\,\| \ \,\|\,$ denotes a chosen norm on $\,T_{x_0}X$.
 However, by assumption up to subsequence $\,f_n\,$ converges
 uniformly on compact subsets to a holomorphic map from $\,\Delta\,$ to $\,X$,
giving  a contradiction.
Thus $\,X\,$ is Kobayashi hyperbolic.  

Finally, let $\,X\cup \{\infty \}\,$ be the Alexandroff compactification 
of $\,X$. As a consequence of Theorem  \ref{ABATE}, 
up to subsequence every  sequence of holomorphic maps 
$\,f_n : \Delta \rightarrow X\,$ converges uniformly on compact subsets  either to the constant map of value $\,{\infty}\,$  or there exists 
$\,\zeta_0 \in \Delta \,$ and $\,x_0 \in X\,$
such that $\,f_n(\zeta_0) \to x_0\,$. In the latter case (ii) implies that
there exists a  subsequence converging uniformly on compact subsets of $\,\Delta$. 
Hence  $\,X\,$ is taut.
\end{proof}  

\nbigskip
As a corollary to Proposition \ref{CHAR}, one has
\bigskip


\begin{cor}
\label{SUBLEVEL}
Let $\,\alpha\,$ be a plurisubharmonic,  continuous
function on a taut manifold $\,X$. Then the sublevel
sets of $\,\alpha\,$ are taut. 
\end{cor}

\begin{proof}
For $\,C \in \R\,$ consider the sublevel set $\,O_C  = \{\, x \in X\ :\  \alpha(x) < C \,\}$
and let $\,f_n : \Delta \rightarrow O_C\,$ be a sequence of holomorphic maps 
such that $\,f_n(\zeta_0) \rightarrow x_0 \,$, for some 
$\,\zeta_0 \in \Delta\,$ and $\,x_0 \in O_C$. Since $\,X\,$ is taut, 
Proposition \ref{CHAR} applies to show that up to subsequence 
 $\,f_{n}\,$ converges uniformly on compact subsets to
 a holomorphic map  $\,f:\Delta \rightarrow X$. Note that $\,\alpha\circ f (\zeta_0) < C\,$
and  by continuity  $\,\alpha \circ f \leq C\,$ on $\,\Delta$. Then the
maximum principle for plurisubharmonic functions implies 
that  $\,\alpha \circ f <C\,$
on $\,\Delta$, i.e.
$\,f(\Delta) \subset O_C$. Finally the statement follows from 
Proposition \ref{CHAR}.
  \end{proof}

  
   \bigskip

Next we recall  two results due to D.D. Thai and N. L. Huong.
(\cite{ThHu}, Lemma\,3 and Cor.\,4).
For  analogous statements where tautness is replaced by
 hyperbolicity or complete hyperbolicity,  
see  \cite{Kob}, Thm 3.2.8.

\bigskip
\begin{pro}
\label{PREIM}
Let $\,X\,$ and $\,Y\,$ be complex manifolds and $ \,F: X \to Y\,$ a holomorphic  map.
If   $\,Y\,$ is  taut and  admits an open covering 
$\, \{ U_{j} \} \,$  such that $\,F^{-1}(U_j)\,$ is taut for all $\,j$,  then  $\,X\,$ is taut. 
\end{pro}


\bigskip
\begin{pro}
\label{COVERING}
Let $\,X\,$ and $\,Y\,$ be complex manifolds and $ \,F: X \to Y\,$ be a holomorphic covering.
Then $\,Y\,$ is  taut if and only if so is $\,X\,$ 
\end{pro}

\nbigskip
For later use we also collect the following well-known facts.

\bigskip
\begin{lem}
\label{POTENTIAL}
 Let $\,\theta\,$ be a real,  positive
and closed  $\,(1,1)$-current on a complex manifold $\,X$. 

 \medskip
\item{\rm (i)} If $\,H^1(X,{\mathcal O})=H^2(X,\R)= 0$, then there exists a plurisubharmonic function $\,\tau\,$  on $\,X\,$  such that
$\,\theta = i \partial \bar{\partial} \tau$.

\medskip
\item{\rm (ii)}   If $\,X\,$ is compact K\"ahler and $\,\theta\,$ is exact then $\,\theta = 0$.

\medskip
\item{\rm (iii)}  If $\,H^1(X,\R)= 0\,$ and $\,\tau\,$ is a pluriharmonic 
function on $\,X$, then there exists a holomorphic function $\,f : X \rightarrow\C\,$ such that  $ \,{\rm Im}\, f = \tau$.

\end{lem}

\begin{proof}
(i) follows from the proof of Prop. III 1.19 in  \cite{Dem}.
For (ii) note that (i) implies that  there exist a locally finite  open covering
$\,\{ U_j \}\,$ of $\,X\,$ and plurisubharmonic functions $\,\tau_j\,$  on $\,U_j\,$
such that $\,\theta|_{U_j} =  i \partial \bar{\partial} \tau_j.$
Let $\,\psi_j\,$ be a partition of unity associated to $\,\{ U_j \}$, define  
$ \,T :=\sum_{j} \psi_j \tau_j\,$ and $ \,\Theta := \theta - i \partial \bar{\partial} T.$
Then for $\,j_0\,$ fixed one has 
$ \Theta |_{U_{j_0}} = (\theta - i \partial \bar{\partial} T)|_{U_{j_0}}
 =i \partial \bar{\partial} \sum_{j } \psi_j (\tau_{j_0} - \tau_j).$ Since $\,\tau_{j_0} - \tau_j\,$ is pluriharmonic on $\,U_{j_0} \cap U_j,$ it follows that $\,\Theta\,$
is a smooth, exact, real $\,(1,1)$-form on $\,X.$ Then, the classical  
$\,\partial \bar \partial$-Lemma for compact K\"ahler manifolds
(see e.g. \cite{GrHa},$\,$Lemma$\,$1.2,$\,$p.$\,$148) implies that
there exists a smooth function $\,Q\,$ on $\,X\,$ such that 
$\,\Theta = i \partial \bar \partial Q$. 
Hence  $\,\theta =   i\partial \bar \partial (Q + T)\,$
and Thm.$\,$I 3.31
 in \cite{Dem}  implies that $\, Q + T\,$
is plurisubharmonic on $\,X$. Since $\,X\,$ is compact,
$\,Q + T\,$ is constant and  consequently  $\,\theta\,$ is zero.     
For (iii) see \cite{Dem} Theorem I 5.16.
\end{proof}

\bigskip

Let us briefly recall the notion of globalization in the context 
of  $\,\R$-manifolds. For further details
and generalizations we refer to \cite{Pal},
\cite{HeIa}, \cite{CIT} and \cite{MiOe}. 
An $\,\R$-action by biholomorphisms on a complex manifold $\,X\,$ induces
a local holomorphic $\,\C$-action by integration of the associated
holomorphic vector field.
This is given by 
an open 
neighborhood $\,\Sigma\,$  
of the neutral  section $\,\{e\}\times X\,$ in $\,\C
\times X\,$ and 
a holomorphic  map 
$\, \Phi:\Sigma \to X $, $\,(\lambda, x) \to \lambda \cdot x$,
 such that

\nmedskip
(i) the set $\,\{\lambda \in \C \ :\ (\lambda,x)\in \Sigma\} \,$ is 
connected for all $\,x\in X$,

\nmedskip
(ii) for all $\,x\in X\,$ one has $\,0 \cdot x=x$,

\nmedskip
 (iii) if $\ (\mu +\lambda,x)\in \Sigma$, $\ (\lambda,x)\in \Sigma\ $ and 
$\ (\mu, \lambda \cdot x)\in 
\Sigma$, $\ $then $\ (\mu +\lambda)\cdot x=\mu \cdot(\lambda \cdot x)$.

\medskip

A possibly
non-Hausdorff complex manifold with a global $\,\C$-action containing
$\,X\,$ as an  $\,\R$-invariant domain is called a {\it
globalization\/} of the local $\,\C$-action. By \cite{HeIa}, if $\,X\,$ is holomorphically
separable  there exists a (unique) {\it universal} globalization $\,X^*$.
That is, a globalization 
with the following  universal property: for any $\,\R$-equivariant holomorphic map
$\,f : X \to Y \,$ into a $\,\C$-manifold
 there exists a $\,\C$-equivariant holomorphic extension $\,f^* : X^* \to  Y$.

 For $\,x\,$ in $\,X^*\,$
let  $\,\Sigma_x=\{ \,\lambda \in \C \ : \lambda \cdot x  \in X \, \}$. Then
$\,\Sigma_x\,$ is $\,\R$-invariant, connected  and there exist
upper semicontinuous  functions 
$\,\alpha, \beta:X^*  \to \R \cup \{-\infty\}\,$ defined by
$$\,\Sigma_x=\{ \,\lambda \in \C \ : \alpha(x) < {\rm Im}\, \lambda <-\beta(x) \, \}\,.$$
Note that  $\,\alpha \,$ and $\,\beta\,$ are $\,\R$-invariant and  $\,\alpha + \beta < 0$.
Moreover, an element  $\,x\,$ of  $\,X^*\,$ belongs to $\,X\,$ if and only if
$\,\alpha(x)  < 0 < - \beta(x)$. Thus 
$$\,X = \{ x \in X^* : \alpha(x)  < 0 \ {\rm and } \ \beta(x)  < 0 \}.$$
We recall  the basic  properties
of $\,\alpha\,$ and $\,\beta\,$ in the case when $\,X\,$ is a 
taut, Stein manifold.


\bigskip
\begin{lem}
\label{ALPHABETA}
Let  $\,X\,$ be a taut, Stein $\,\R$-manifold. Then

\medskip
\item{\rm (i)} the functions $\,\alpha\,$ and $\,\beta\,$ are continuous and  plurisubharmonic,

\medskip
\item{\rm (ii)} for  $\, \lambda  \in \C \,$ and $\,  \ x \in X^* \,$ one has   
$$\alpha(\lambda \cdot x) = -{\rm Im}\,(\lambda) + \alpha(x) \quad \quad   \quad  \beta(\lambda \cdot x) =  {\rm Im} \,(\lambda) + \beta(x)\,.$$

\medskip
\item{\rm (iii)} the sum  $\,\alpha\,+\,\beta\,$ is a negative, $\,\C$-invariant, plurisubharmonic, continuous 
function,
 
\medskip
\item{\rm (iv)} if the $\,\R$-action is proper, then $\, \max(\alpha(x),\beta(x)) > - \infty \,$
for all $\,x\,$ in $\,X^*$.

\end{lem}

\begin{proof} 
(i) Plurisubharmonicity of   $\,\alpha\,$ and $\,\beta\,$ in the 
case where $\,X\,$ is a Stein $\,\R$-manifold 
is proved in \cite{For}. Since $\,X\,$ is also taut,
such functions are continuous
(\cite{Ian}, \cite{MiOe},\,Prop.\,3.2). (ii)  is a direct consequence of  the definition and
(iii) follows from (i) and (ii). 
For (iv) note that properness of the $\,\R$-action implies that  there are no fixed points. 
Therefore if $\,\alpha(x)=\beta(x)=-\infty\,$ for some $\,x\,$ in 
$\,X$, the (local) $\,\C$-orbit through $\,x\,$ is biholomorphic either to 
$\,\C\,$ or to $\,\C^*$.
 Since $\,X\,$ is taut, this
gives a contradiction. Recalling that $\,X^*=\C\cdot X$, the result follows
from (ii).
\end{proof}

\bigskip
Finally we recall the following result of C. Miebach and K. Oeljeklaus
(see \cite{MiOe}, Thm. 4.4) which is often used in the sequel.


\bigskip
\begin{theorem}
\label{BUNDLE}
Let $\,X\,$ be a  2-dimensional, taut, Stein manifold with  a proper
$\,\R$-action. Then the $\,\C$-action on $\,X^*\,$ is proper, i.e. 
$\,X^*\,$ can be regarded as a holomorphic principal $\,\C$-bundle
over the Riemann surface $\,S:=X^*/\C$. In particular
if $\,S\,$ is non compact,  then $\,X^*\,$ is $\,\C$-equivariantly
biholomorphic to $\,\C \times S$.
\end{theorem}

\bigskip
Note that last part of the statement follows directly from the fact that on a non
compact Riemann surface $\,S\,$ the cohomology group $\,H^1(S, \mathcal O)\,$
vanishes.


\bbigskip


\section{distinguished $\,\R$-invariant domains in $\,\C^2$}
 
 \bigskip

Consider the domains
of $\,\C^2\,$ of the form 
$\,\Omega_{a} = \{\, (z,w) \in \C^2\ :\ a(w) < {\rm Im}\, z\, \}$,
with $\,a: \C \to \{-\infty\} \cup \R\,$ an upper semicontinuous function.
Note that $\,\R\,$ acts properly on $\,\Omega_{a}\,$ by translations 
on the first factor.
The main result of this section is Theorem \ref{TAUTT}, where we determine
necessary and sufficient conditions for $\,\Omega_{a}\,$ to be taut.
  
We already noted in the introduction that if $\,\Omega_a\,$
is taut, then $\,a\,$ is  real valued and subharmonic. 
Moreover  $\,\C^2\,$ is the universal globalization
of $\,\Omega_{a}$, therefore by (i) of Lemma \ref{ALPHABETA}
the function $\,\alpha:\C^2 \to  \{-\infty\} \cup \R$, given by 
$\,(z,w) \to a(w) - {\rm Im}\, z$,   is continuous.
As a consequence $\,a\,$ necessarily belongs to 
 $$\mathcal C:= \{{\rm\, subharmonic,\ real\ valued, \ continuous\ functions\ on \ }
\C\,\}.$$

Also note that if $\,\Omega_{a}\,$ is taut, then for all positive $\,\tau \in \R\,$ the
domain $\,\Omega_{\tau a}\,$ is also taut, since the biholomorphism 
$\,\C^2 \to \C^2$, defined by $\,(z,w) \to  (\tau z, w )$, maps  $\,\Omega_{a}\,$
onto $\,\Omega_{\tau a}$.
Thus the set of interest  $\,\mathcal F := \{ \,a \in \mathcal C \ : \  \Omega_{a} \ {\rm is \ taut}\,\}\,$ is a cone. 
Here we show that $\,\mathcal F\,$ coincides with 
$\, \{ \,a \in \mathcal C \ : \ a\  {\rm is \ not \ harmonic}\,\}$.
We need some preliminary lemmata.


\bigskip
\begin{lem}
\label{SUBSEQ}
Let $ \,a \in \mathcal C\,$ and  $\,(f_n,g_n) : \Delta \rightarrow \Omega_{a}$ be a sequence of holomorphic maps such that 

\medskip
\item{\rm (i)} 
$\,f_n(\zeta_0) \to z_0$,
for some $\,\zeta_0 \in \Delta$ and $\,z_0 \in \C$,

\medskip
\item{\rm (ii)}
  $\,g_n\,$ converges uniformly on compact subsets of $\,\Delta\,$
to a holomorphic map $\,g:\Delta \to \C\,$ such that $\,(z_0,g(\zeta_0) ) \in \Omega_a$.

\nmedskip
Then there exists a subsequence of $\,f_n\,$ 
converging uniformly on compact subsets of $\,\Delta\,$   
to a holomorphic map $\,f:\Delta \to \C\,$ such that $\,(f,g)(\Delta) \subset
\Omega_a$.

\end{lem}

\begin{proof}  
Let $\,U_1\,$ be a relatively compact disk of $\,\Delta\,$ containing $\,\zeta_0$.
By condition (ii)
the sequence $\,g_n\,$ converges uniformly to $\,g\,$ 
on the closure $\, \overline U_1\,$ of $\,U_1$.
Then, for $\,n\,$ large enough $\,f_n(U_1)\,$ is contained in the  set
$$S_1:=\{\, z \in \C \ :\  {\rm Im}\, z > \min_{w \in \overline{U_1}} \{a(g(w))\} -1 \,\} \,,$$
 which is biholomorphic to the unit disc of $\,\C$. In particular
 $\,S_1\,$ is taut, therefore there exists a subsequence $\,f_{n,1}\,$ of $ \,f_n\,$
converging uniformly on compact subsets of $\,U_1\,$ to a holomorphic map
$\, f_1:U_1 \to S_1$.

Complete $\,U_1\,$ to an increasing sequence of simply connected 
domains $\,\{U_k\}_{k \in \N}\,$ which exhaust $\,\Delta$.
By iterating the above argument, for each $\,k \in \N\,$ one obtains subsequences 
$\,\{f_{n,k}\}_{n\in \N}\,$ converging uniformly on compact subsets of $\,U_k\,$
to holomorphic maps
$\, f_k:U_k \to S_k$.
Then the diagonal sequence $\,\{f_{j,j}\}_{j\in \N}\,$
 converges uniformly on compact subsets of $\,\Delta$.
 
 Finally note that $\,  a \circ g(\zeta_0)- {\rm Im}\,f(\zeta_0)<0\,$
and  by continuity  $\, a \circ g-  {\rm Im}\,f \leq 0\,$
on $\,\Delta$. Then, by the
maximum principle for subharmonic functions, 
$\,a \circ g-  {\rm Im}\,f<0\,$ on $\,\Delta$, i.e.
$\,(f,g)(\Delta) \subset
\Omega_a$.

 \end{proof}

\nbigskip

Given a subharmonic function $\,a\,$ on $\,\C$,
denote by $\,M_{\zeta_0,r}(a)\,$
its mean value 
$\,\frac{1}{2 \pi}  \int_0^{2 \pi} a(\zeta_0 + r e^{i \theta})d \theta$.

\bigskip
\begin{lem}
\label{RIFORM} For $\,a\,$ in $\,\mathcal C\,$ the following conditions are equivalent.
 
  \medskip
 \item{\rm (i)} $\,a\,\in\,\mathcal F$,
 
  \medskip
 \item{\rm (ii)} for any   sequence of holomorphic functions
 $\,g_n : \Delta \rightarrow \C\,$
 satisfying 
 
   \medskip
 \item{\ \ }
{\rm (a)} $\,g_n(\zeta_0) \rightarrow w_0\,$ for some $ \,\zeta_0 \in \Delta\,$ and $\, w_0 \in \C\,$,

   \medskip
 \item{\ \ } {\rm (b)} for every 
$\,0 < r < 1- |\zeta_0|\,$ there exists  $\,M_r \in \R\,$  such that 
 \item{\ \ }{\ \ \ \ } $\,M_{\zeta_0,r}(a \circ g_n) < M_r$ for all $\,n \in \N$,   

\nmedskip
there exists a subsequence  converging uniformly on
compact subsets of $\,\Delta_{1-|\zeta_0|}(\zeta_0).$
\end{lem}

\begin{proof}
Assume that $\,\Omega_a\,$ is taut and let $\,g_n\,$ be a sequence 
as in (ii). For $\,n \in  \N\,$ and $\,0 < r < 1-  |\zeta_0|$,  denote by $\,h_n\,$ the 
harmonic function on $\,\Delta_{r}(\zeta_0)\,$ which coincides with $\,a \circ
g_n\,$ on the boundary of $\,\Delta_{r}(\zeta_0)$. Then 
$\,h_n(\zeta_0) =M_{\zeta_0,r}(a \circ g_n)\,$
and consequently, for  $\,n\,$ large enough, one has   
$$a(w_0) - 1 < a(g_n(\zeta_0)) \leq h_n(\zeta_0) < M_r.$$ 
As a consequence, up to subsequence $\,h_n(\zeta_0)\,$ converges to
a real number $\, y$.
 Let  $\,f_n : \Delta_r(\zeta_0)\rightarrow \C\,$ be the sequence of holomorphic functions
 defined by 
$\,{\rm Im}\, f_n = h_n + 1\,$ and $\,{\rm Re} \, f_n(\zeta_0) = 0$.
Since $\,{\rm Im}\, f_n = h_n + 1 \ge  a \circ g_n + 1 > a \circ g_n$,
it follows that $\,(f_n,g_n)\,$ defines a sequence of holomorphic maps
from $\,\Delta_{r}(\zeta_0)\,$ to $\Omega_{a}$.

Moreover $\,(f_n,g_n)(\zeta_0) \to (i(y+1), w_0) \in \Omega_a \,$
and $\, \Omega_a\,$
is taut. Then by Lemma \ref{CHAR} there exists a subsequence 
$\,(f_n,g_n)\,$ converging uniformly
on compact subsets of $\,\Delta_{r}(\zeta_0)$. 

Let $\,r_k\,$ be an increasing sequence of positive numbers converging to 
$\,1-  |\zeta_0|\,$ such that $\,r_1=r$.  The analogous argument 
as above shows that
there exist subsequences $\,(f_{n,k},g_{n,k})\,$ 
converging uniformly
on compact subsets of $\,\Delta_{r_k}(\zeta_0)$. Then 
the diagonal subsequence   $\,(f_{n,n},g_{n,n})\,$ converges uniformly on compact
subsets of $\,\Delta_{1-|\zeta_0|}(\zeta_0)$ and so does $\,g_{n,n}$. This implies (ii).

Conversely assume (ii) and  let  $\,(f_n,g_n) :\Delta \rightarrow \Omega_{a}\,$
be a sequence of holomorphic maps such that 
$\,(f_n,g_n)(\zeta_0) \rightarrow (z_0,w_0)$, for some $\,\zeta_0\,$ in $\,\Delta\,$
and $\,(z_0,w_0)\,$ in $\, \Omega_a$. By Lemma \ref{CHAR} it is enough to 
show that, up to subsequence,  $\,(f_n,g_n)\,$ 
converges uniformly
on compact subsets of $\,\Delta\,$ to some $\,(f,g)\,$ with 
$\,(f,g)(\Delta) \subset \Omega_a$.
 Note that for $\,0 < r < 1 -|z_0|\,$ and  $\,n\,$ large enough one has 
 $$\,{\rm Im} \, z_0 + 1 > {\rm Im} \,f_n(\zeta_0)
 =  M_{\zeta_0,r} ({\rm Im} \, f_n) >M_{\zeta_0,r} (a \circ g_n)\,.$$
Thus, by assumption,
  up to subsequence $\,g_n\,$ converges uniformly on compact sets of the disk
  $\,\Delta_{1-|z_0|}(z_0)\,$ and, by Lemma \ref{SUBSEQ} so does 
  $\,f_n.$ Therefore for every point $\,\zeta \in   \Delta_{1-|z_0|}(z_0)\,$
  there exists a subsequence of $\,(f_n,g_n)\,$ converging at $\,\zeta\,$
  to an element of $\,\Omega_a$.  Then
 by constructing a finite chain of disks one shows
  that, up to subsequence, $\,(f_n,g_n)\,$ converges at $\,0\,$ to an element
  of $\,\Omega_a$.
  Finally the analogous argument as above 
  implies that, up to  subsequence $\,(f_n,g_n)$, 
  converges uniformly
on compact subsets of  $\,\Delta\,$ to some $\,(f,g)\,$ with 
$\,(f,g)(\Delta) \subset \Omega_a$.
\end{proof} 


\bigskip
\begin{lem}
\label{THECONE}
The cone $\,\mathcal{F}\,$ has the following properties.

  \medskip
 \item{\rm (i)}  Harmonic functions do not belong to $\, \mathcal{F}$.
 
   \medskip
 \item{\rm (ii)}\,\ If $\, a \in \mathcal{C}\,$ is non constant  and bounded from below,
  then $\,a \in \mathcal{F}$.
 
   \medskip
 \item{\rm (iii)} If $ \,b\,  \in \mathcal{C}\,$ and  $ \, c  \in \mathcal{F}\,$ then 
   $ \, b + c \in \mathcal{F}$.
\end{lem}

\begin{proof}
(i) If $\,a\,$ is harmonic, then $\,a = {\rm Im}\, f\,$ 
for some holomorphic $\,f : \C \rightarrow \C$.
Then the biholomorphism of $\,\C^2\,$ defined by
$\,(z,w) \rightarrow (z- f(w),w)\,$ maps
$\,\Omega_{a}\,$ onto $\, \{\,(z,w) \in \C^2 \ : \  {\rm Im}\, z > 0 \,\}$,
which is not taut. Thus $\,\Omega_{a}\,$ is not taut.

For (ii)  consider the restriction to $\,\Omega_{a}\,$ of the projection from $\,\C^2\,$ onto the first factor given by 
$$\,p|_{\Omega_{a}}:\Omega_{a}
\to p(\Omega_{a})\,, \quad \quad (z,w) \to z\,.$$
Since $\,a\,$ is bounded from below, the image $\,p(\Omega_{a})\,$ is contained in 
the half plane $\,\{ \,{\rm Im}\, z > \inf_\C a \,\}$, which 
 is taut. Then,  by Lemma \ref{PREIM}, in order to prove that 
$\,\Omega_{a}\,$ is taut it is enough to show that 
$\,(p|_{\Omega_{a}})^{-1}(U) \,$
is taut for every  relatively
compact open subset $\,U\,$ in $\,p(\Omega_{a})$.

For this, let $\,M\,$ be the maximum of  $\,{\rm Im}\, z\,$  on the closure of
$\,U\,$  and note that $\,(p|_{\Omega_{a}})^{-1}(U) \,$ is contained in 
$\,U \times \{\,a<M\,\}$.  Since $\,a\,$ is not constant, it is not bounded. 
As a consequence $\,\{\,a<M\,\}\,$ is a hyperbolic domain of $\,\C$. Thus it is taut
and so is $\,U \times \{\,a<M\,\}$. Finally, the image $\,(p|_{\Omega_{a}})^{-1}(U) \,$ 
is the zero sublevel set in $\,U \times \{\,a<M\,\}\,$ of the subharmonic,
continuous function $\,(z,w) \to a(w)- {\rm Im}\, z$. Thus it is taut
by Corollary \ref{SUBLEVEL}, concluding (ii).

For (iii) let $\,g_n : \Delta \rightarrow \C\,$  be a sequence of
holomorphic maps such that 
$\,g_n(\zeta_0) \rightarrow w_0\,$ for some $ \zeta_0 \in \Delta$, $\, w_0 \in \C\,$
and for every
$\,0 < r < 1- |\zeta_0|\,$ there exists a real number $\,M_r\,$  such that 
$M_{\zeta_0,r}((b+c) \circ g_n) < M_r\,$ for all $\,n\in \N$.
Then by Lemma \ref{RIFORM} in order to show that $\,b+c\,$ belongs to
$\,\mathcal F$, it is enough to  find a  subsequence of $\,g_n \,$ converging
uniformly on compact subsets of $\,\Delta$.
For $\,0 < r < 1- |\zeta_0|\,$ and $\,n\,$ large enough one has  
$$M_r > M_{\zeta_0,r}((b+c) \circ g_n)
\geq 
b(g_n(\zeta_0)) + M_{\zeta_0,r}(c \circ g_n) 
> 
b(w_0)-1 + M_{\zeta_0,r}(c \circ g_n)\,.$$
Hence 
$$M_{\zeta_0,r}(c \circ g_n) < M_r -b(w_0)+1\,.$$

\nsmallskip
Since $\,c \in \mathcal{F}$, Lemma \ref{RIFORM} implies that there exists 
 a  subsequence of $\,g_n \,$ converging uniformly on compact subsets of
$\,\Delta$, as wished.
\end{proof}

\bigskip

\begin{theorem}
\label{TAUTT}
Let $\,a: \C \to \{- \infty \} \cup \R\,$ be an  upper semicontinuous function.
Then $\,\Omega_a:= \{\,(z,w) \in \C^2 \ : \ a(w)<{\rm Im}\,z \}\,$ is taut if and 
only if $\,a\,$ is a real valued, subharmonic, non-harmonic, continuous  function.

\end{theorem}

\begin{proof} We already noted at the beginning of the section
 that if $\,\Omega_a\,$ is taut, then 
$\,a\,$ belongs to $\,\mathcal C$. Moreover,
by (i) of the above lemma $\,a\,$ is not
harmonic, giving one implication.

Conversely, given $\,a \in \mathcal C\,$ non-harmonic
we want to show that $\,a \in \mathcal F$.
By (ii) and (iii) of the above lemma, it is enough to show that $\,a=b+c$, with 
$ b, \ c \in  \mathcal{C} $ and  $\,c \,$ non constant 
and bounded from below. 

For this consider the positive measure $\,\mu=L(a)$, where $\,L(a)\,$ denotes
the laplacian
of $\,a$, and choose $\,r\,$ big enough such that $\,\mu\,$
is non zero on $\,\Delta_r(0)$. Let $\,\chi_{\Delta_r(0)}\,$ be the
characteristic function of  $\,\Delta_r(0)\,$ and define 
 $\,\mu_1 = (1- \chi_{\Delta_r(0)})\mu\,$ and 
 $\,\mu_2 = \chi_{\Delta_r(0)} \mu$, so that $\,\mu = \mu_1 +\mu_2\,$ 
 gives a decomposition of  $\,\mu\,$ as 
a sum of positive measure on $\,\C$. 
Note that $\,\mu_2\,$ is non zero with compact support and
consider the potential $\,c:\C \to \R \cup\{-\infty\}\,$  associated to  $\,\mu_2\,$
defined by 

$$c(w) := \frac{1}{2 \pi} \int_{\C} \log(|w-\xi |) d \mu_2(\xi) = \frac{1}{2 \pi} \int_{\Delta_r(0)} \log(|w-\xi |) d \mu_2(\xi)\,.$$
Then the laplacian $\,L(c)\,$ of $\,c\,$ coincides with $\,\mu_2\,$
(see$\,$e.g.$\,$\cite{Kli},$\,$Prop.$\,$4.1.2),
therefore $\,c\,$ is non constant and subharmonic.

Furthermore, the real $\,(1,1)$-current $\,\mu_1 d\xi d\overline \xi\,$
is closed and  positive on $\,\C$, hence by (i) of Lemma \ref{POTENTIAL}
there exists a subharmonic
function $\,\tilde b:\C \to \R \cup\{-\infty\}\,$ such that $\,L(\tilde b) =\mu_2$.
It follows that $\,L(\tilde b + c) = L(a)\,$ and consequently
$\,a =  \tilde b + c +h$, with $\,h\,$ harmonic on $\,\C$.
This implies that  $\, \tilde b + c\,$ is continuous and real valued.
Since $\, \tilde b\,$ and  $\, c\,$ are everywhere smaller than $\,+\infty$,
they are also real valued. Moreover $\,c\,$ is  upper semicontinuous,
$\,-\tilde b\,$ is  lower  semicontinuous and
$\,c=  -\tilde b +a-h$, with $\,a-h\,$ continuous. Thus
$\, \tilde b\,$ and  $\, c\,$ are continuous subharmonic functions, i.e.
they belong to 
$\,\mathcal C$, and
 $\,b:= \tilde b +h\, \in \,\mathcal C$.

Finally note that the non constant function $\,c\,$ is 
bounded from below. Indeed by definition of $\,c$,
if $\,w\,$ is not in $\,\Delta_{r+1}(0)\,$ then  $\,c(w) \geq 0$.
Since $\,c\,$ is continuous, this implies that
$\,c \geq \min\{0,m \} $, with $\,m:= \min_{w \in \overline { \Delta_{r+1}(0)}} \{c(w)\}$.
 Then $\,a=b+c\,$ gives the desired 
decomposition.
\end{proof}

\bbigskip


\section{models with compact base}
 
 \bigskip

Let $\,S\,$ be a compact hyperbolic Riemann surface, say 
$\,S = \Delta/ \Gamma$, with $\,\Gamma\,$ the subgroup in 
 $\,\Aut(\Delta)\,$ of deck transformations of the universal covering 
 $\,\Delta \to S$. Choose a non trivial group homomorphism
$\,\Psi: \Gamma \to \R\,$ and let $\,\Gamma\,$ act on 
$\,\C \times \Delta\,$ by $\,\gamma \cdot (z,w):=(z+\Psi(\gamma), \gamma \cdot w)$.
Endow the quotient $\,\C \times \Delta/\Gamma\,$ with the $\,\R$-action
defined by $\,t \cdot (z,w) := (z+t, w)$. We introduce the first class
of models as  $\,\R$-invariant
subdomains of $\,\C \times \Delta/\Gamma$.

\nbigskip
{\bf Type CH} $\ $A model of type CH with compact hyperbolic base 
base $\,S= \Delta/\Gamma\,$ is given by
$$\,H \times \Delta/\Gamma\,,$$
 where $\,H\,$ is a proper, $\,\R$-invariant, connected strip
of $\,\C$. Up to $\,\R$-equivariant biholomorphism, we may assume that 
$\,H \,$ is one of the strips
$\, \{0< {\rm Im}\,z\}$, $\, \{ {\rm Im}\,z<0\}\,$ or $\, \{ 0<{\rm Im}\,z<C\}$,
for some real positive $\,C$. 

\bigskip
\begin{pro}
\label{TYPECH}
Let $\,X\,$ be a model of type {\rm CH} with base $\,S= \Delta/ \Gamma$. Then

\smallskip
\item{\rm (i)} the universal globalization of $\,X\,$ is $\,\C \times \Delta/\Gamma$, which 
is Stein.

\smallskip
\item{\rm (ii)}  $\,X\,$ is a taut, Stein manifold with a proper $\,\R$-action.

\end{pro}
 \bigskip
Before proving the above proposition we need a preparatory lemma. 
Given  a rank two holomorphic vector bundle $\,E\,$ over a compact Riemann
surface $\,S$, denote by   $\,P\,$ its (fiberwise) projectification
and let  $\,p: E \setminus S \rightarrow P\,$ be the canonical projection.
Here $\,S\,$ is identified with the zero section in $\,E$.
 Let  $\,\sigma: S \rightarrow P\,$ be  a holomorphic section
 of $\,P\,$ and consider its image $\,C:=\sigma(S)$. 
Recall that the normal  bundle $\,N\,$
of  the curve $\,C\,$ is given by  $\,TP |_{C}/ T C\,$
and it can be identified with the line bundle  $\,\sigma^*(N)\,$
over $\,S$. 

Regard  the tautological line bundle $\,{\mathcal O}(-1)\,$
as a subbundle in  $\, \pi^*(E)$,  where  $\,\pi:P \to S\,$ is the 
bundle projection. Then the holomorphic line bundle 
associated to $\,\sigma\,$ is  $\,L := \sigma^*({\mathcal O}(-1))$ and can
be identified with the subbundle of $\,E\,$ given by $\,p^{-1}(C) \cup S$.

\begin{lem}
\label{NORMAL}
The normal bundle  $\,\sigma^*(N)\,$ is isomorphic to $\,(E/L) \otimes L^*.$ 
\end{lem}

\begin{proof}
Consider the relative tangent bundle 
$\, T_{P/S} := {\rm Ker}\, d  \pi$.  We first note that $\,N\,$
is isomorphic to the restriction
 $\, T_{P/S}|_{C}\,$ of such a bundle to $\,C$, since one has  
 the short exact sequence  of vector bundles over $\,C\,$
$$\,0 \rightarrow T C \rightarrow TP |_{C}
 \rightarrow T_{P/S}|_{C} \rightarrow 0\,,$$
  where the third map is defined by 
$\,v \mapsto v - d \sigma \circ d \pi(v).$

We first assume that  $\,L\,$ is trivial,
i.e. it admits a non zero holomorphic section $\tau.$ Then one 
has  the commutative diagram
$$
\xymatrix{
E\setminus S \ar[rr]^p & & P \\
 & S \ar[ul]^\tau \ar[ur]_\sigma &}
$$
and an exact sequence of vector bundles over $\,\tau(S)\,$
$$ 0 \rightarrow T_{L/S}|_{\tau(S)} \rightarrow T_{E /S}|_{\tau(S)} 
\rightarrow   p^*({T_{P/S}}|_{C}) \rightarrow 0\,,$$
where the third map is given by $\,v \mapsto dp(v)$.
Since $\,p \circ \tau =\sigma$, by applying  $\,\tau^*\,$ one obtains the following exact sequence of vector
 bundles over $\,S\,$
$$0 \rightarrow L \rightarrow E \rightarrow  \sigma^*({T_{P/S}}|_{C} ) \rightarrow 0\,,$$
where we use the natural identification $\,\tau^*(T_{F/S}|_{\tau(S)})
\cong F\,$ for any vector subbundle $\,F\,$ of $\,E$.
Moreover, by recalling that $\,N\,$ is isomorphic to
$\, T_{P/S}|_{C}\,$, one obtains  that $\,\sigma^*(N)\,$ is isomorphic to  $\, E/L$, as wished.

 Finally, if $\,L\,$
 is non trivial note that  $\,P \,$ can be regarded as the projectification
of $\,E \otimes L^*\,$ and in this realization $\,\sigma(S)\,$ is the projectification
of the trivial line bundle $\,L \otimes L^*$. 
Then an analogous argument as above
implies that $\,\sigma^*(N)\,$ is isomorphic to $\,E \otimes L^*/L \otimes L^*\,$
and  by the exactness of the sequence of vector bundles over $\,S\,$
$$\,0 \rightarrow L \otimes L^* \rightarrow E \otimes
L^* \rightarrow (E/L) \otimes L^* \to 0\,.$$
one has 
$\,E \otimes L^*/L \otimes L^* \cong (E/L) \otimes L^*.$
\end{proof}
 
\nbigskip
{\it Proof of Proposition \ref{TYPECH}}
(i) Note that $\,X\,$ is orbit-connected in $\,\C \times \Delta/\Gamma$. Then
Lemma$\,$1.5 in \cite{CIT} implies that $\,X^*:=\C \times \Delta/\Gamma\,$
is the universal globalization of $\,X$.
Consider the $\,\P^1$-bundle  
$\,P:=\P^1 \times \Delta/\Gamma$, where
 $\,\Gamma\,$ act on $\,\P^1 \times \Delta \,$ by 
$\,\gamma \cdot ([z_1:z_2],w):= ([z_1+\Psi(\gamma)z_2:z_2],\gamma \cdot w)$.
Then $\,X^*\,$ is embedded in $\,P\,$ via the map 
$$[z,w] \to [[z:1],w] \,.$$
and the union of  points at infinity defines the complex curve
$\, C:=\{\, [[1:0],w] \in P\ : \ w \in \Delta \,\}\,$ which is 
biholomorphic to $\,S$. Indeed it can be regarded as the holomorphic section
$\,\sigma:S \to P$, defined by $\, [w] \to  [[1:0],w]$.

We wish to apply Theorem 1, p. 590 in \cite{Ued} in order to obtain a suitable
strictly plurisubharmonic function on $\,V_0 \setminus C$, for some  open
neighborhood $\,V_0\,$ of $\,C\,$ in $\,P$. For this we first check that 
the normal bundle of $\,C\,$ is trivial. Consider the rank two
vector bundle over $\,S\,$ defined 
by $\,E:= \C^2 \times \Delta /\Gamma$, where $\,\Gamma\,$ acts 
on $\,\C^2 \times \Delta\,$ by 
$\,\gamma \cdot ((z_1,z_2),w):= ((z_1+\Psi(\gamma)z_2,z_2),
\gamma \cdot w)$.
Note that the line subbundle $\,L:=\{\,[(z_1,z_2),w] \in E \ : z_2=0 \,\}\,$
associated to the section $\,\sigma\,$
is trivial. Indeed it admits the global section $\,[w] \to  [(1,0),w]$.
Since $\,P\,$ is the projectification of $\,E$,  by Lemma \ref{NORMAL}
this implies that the
normal bundle of $\,C:= \sigma(S)\,$  is isomorphic to $\,E/L$.
Moreover one has the short exact sequence of vector bundles
over $\,S\,$
$$0 \rightarrow L \rightarrow E \rightarrow 
\C \times S \rightarrow 0\,,$$
where the third map is defined by $\,[(z_1,z_2),w] \to (z_2, [w])$.
Therefore $\,E/L\,$ is trivial and so is the
normal bundle of $\,C$.

Next we check that the curve $\,C\,$ is of type 1, in the sense of  Definition
p. 589 in \cite{Ued}. For this choose   an open covering $\,\{U_j\}\,$ of $\,S\,$
such that there exist injective, local
sections $\,s_j:U_j \to \Delta\,$ of the universal covering $\,\Delta \to S$.
Define local trivializations of $\,P\,$  by 
$$\P^1 \times U_j \to P\,, \quad ([z_1:z_2],p) \to [[z_1:z_2],s_j(p)]\,.$$
Note that  the curve $\,C\,$ is locally defined  by $\,\{z_2=0\}\,$
and in a neighborhood of $\,C\,$ the intersection of two trivializations
associated to  the sections $\,s_j\,$ and $\,s_k\,$
is given by
$$[[1:z_2],s_k(p)]=[[1:z'_2],s_j(p)].$$
This implies that there exists $\,\gamma \in \Gamma\,$ such that  
$\,s_j(p)= \gamma \cdot s_k(p)\,$ and consequently 
$$[[1:z_2],s_k(p)]=[[1:z'_2],\gamma \cdot s_k(p)]
=[[1- \Psi(\gamma) z'_2:z'_2],s_k(p)]\,.$$
Since $\,\Gamma\,$ acts freely on $\,\Delta$, it follows  that $\,z_2 = z'_2/(1-\Psi(\gamma)z'_2)\,$ and 
$$z_2-z'_2= z'_2(\frac{1}{1-\Psi(\gamma)z'_2} -1)=
 (z'_2)^2 \frac{\Psi(\gamma)}{1-\Psi(\gamma)z'_2}=
(z'_2)^2(\Psi(\gamma) + o(z'_2))\,.$$
In our setting the normal bundle of $\,C\,$ is holomorphically trivial,
therefore the locally constant maps $\,f_{jk}:U_j \cap U_k \to \C$, given by $\,p \to \Psi(\gamma)$,
define a cocycle in $\,H^1(S,\mathcal O)\,$ (cf. \cite{Ued}, p. 588).

\nsmallskip
 {\it Claim.} $\,$The cocycle $\,f_{jk}\,$ is cohomologous to zero if and only if
 $\,\Psi\,$ is trivial.
 
\nsmallskip
{\it Proof of Claim.} $\,$By using the above defined sections $\,s_j:U_j \to \Delta\,$
one has 
local trivializations of $\,X^*\,$  given by  
$$\C \times U_j \to X^*\,, \quad (z,p) \to [z,s_j(p)]\,.$$
It follows  that $\,f_{jk}\,$ is the cocycle defining $\,X^*\,$ as a
holomorphic principal $\,\C$-bundle over $\,S$. 
Assume that there exists a holomorphic ($\,\C$-equivariant) 
trivialization  $\,F:X^*  \to \C \times S$. We can choose a
($\,\C$-equivariant) lifting $\,\tilde F: \C \times \Delta \to \C \times \Delta\,$ to the universal coverings  such that $\,\tilde F(z,w)=(z + \tilde f(w), w)$,
with $\,\tilde f:\C \to \C\,$
holomorphic.
Moreover for every $\,\gamma \in \Gamma\,$  one has 
$$\, \tilde F(\gamma \cdot (z,w)) = \gamma \cdot \tilde F (z,w) = (z + \tilde f(w)
,\gamma \cdot w)\,, $$
implying that 
 $\,\tilde f( \gamma \cdot w) + \Psi(\gamma)=\tilde f( w)$.
  In particular
 $$\tilde f(w)-\tilde f( \gamma \cdot w) =\Psi(\gamma) \in \R.$$
Hence $\,{\rm Im}\, \tilde f\,$ is $\,\Gamma$-invariant, therefore 
it pushes down to a harmonic function on  $\,S:=\Delta/\Gamma$. 
Then the compactness of  $\,S\,$ implies that $\,{\rm Im}\, \tilde f\,$  is constant
and consequently $\,\tilde f\,$ is constant. Hence $\,\Psi(\gamma)=0\,$
for all $\,\gamma \in \Gamma$, proving the claim.
\smallskip

Since  $\,\Psi\,$ is non-trivial by assumption, the cocycle $\,f_{jk}\,$ is not cohomologous to
zero, i.e.  the curve $\,C\,$ is of type 1.
Then, by Theorem 1, p. 590 in \cite{Ued} there exists an open neighborhood
$\,V_0\,$ of $\,C\,$ in $\,P\,$ and  a smooth, strictly plurisubharmonic 
function $\,\rho\,$ defined on $\,V_0 \setminus C\,$ such that $\,\lim \rho(p)=\infty\,$
for $\,p\,$ approaching  $\,C$. In particular we may assume that 
$\,\rho\,$ is positive.

Fix $\,N\,$ large enough such that the domain
$\,X_N:=\{\,[z,w] \,\in \C \times \Delta/\Gamma \ : \ {\rm Im}\, z >N\}\,$ is contained in $\,V_0$. Note that  $\,X_N\,$ is Stein, since it admits the
smooth, strictly plurisubharmonic exhaustion  $\,\rho \,+\, \frac{1}{{\rm Im}\, z -N}$.
Moreover for all $\,n \in \N\,$ the domains $\,X_{N-n}\,$ are also Stein,
being biholomorphic to  $\,X_N\,$ via a translation in the first factor.
Furthermore $\,X_{N-n}\,$ can be regarded as a sublevel set
of  the plurisubharmonic function $\,{\rm Im}\,z$, therefore it is Runge in 
$\,X_{N-(n+1)}$. Then 
$\,\C \times \Delta/\Gamma= \cup_n X_{N-n}\,$ is Stein by a classical
result of K. Stein \cite{Ste}.

(ii) Note that $\,X\,$ is an $\,R$-invariant, locally Stein domain in 
the Stein, principal $\,\C$-bundle $\,X^*=\C \times \Delta/\Gamma\,$
over $\,S$.
Thus the $\,\R$-action on $\,X\,$ is proper and $\,X\,$ is Stein by \cite{DoGr}. 
Finally the universal covering of $\,X\,$ is given by
$\,H\times \Delta$, which is taut. Thus $\,X\,$ is taut by Proposition 
\ref{COVERING}.
\qed

\bigskip
\begin{rem}
It was pointed out to us by Christian Miebach that a similar 
strategy as above applies to show that every non trivial  
principal $\,\C$-bundle over a compact Riemann surface is Stein.
\end{rem}

\bigskip
\begin{rem}
 \label{BIHOLOCH}
Let $\,F:  H \times \Delta/\Gamma \to
H' \times \Delta/\Gamma'\,$ be an $\,\R$-equivariant biholomorphism 
between two models of type CH and consider a
holomorphic lifting   $\,\tilde F :H \times \Delta \to H' \times \Delta\,$
to the universal covering spaces. 
We claim that  $\,\tilde F(z,w) = (z +  r, \tilde \varphi(w))$, where  $\,r \in \R\,$
and $\,\tilde \varphi \in Aut(\Delta)$. In particular  $\,H = H'$.

In order to prove this, note that  $\,\tilde F\,$ is also $\,\R$-equivariant.
Therefore it induces a biholomorphism $\,\varphi:\Delta/\Gamma \to 
\Delta/\Gamma'$.
As a consequence 
$ \tilde{F}(z,w) = (z + f(w),\tilde \varphi(w))$, with
$\,f: \Delta \rightarrow \C\,$  holomorphic
and $\,\tilde \varphi  \in Aut(\Delta)\,$  a lifting of $\,\varphi\,$ 
with   $\,\Gamma' =  \tilde \varphi \Gamma \tilde  \varphi^{-1}$.
Since  the actions of $\,\Gamma\,$ and $\,\Gamma'\,$  on $\,\C \times \Delta\,$ are given respectively by $\,\gamma \cdot (z,w) = ( z + \Psi(\gamma),\gamma(w))\,$  and 
$\,\gamma' \cdot (z,w) = ( z + \Psi'(\gamma'),\gamma'(w))$, 
 for $\,\gamma \in \Gamma\,$ one has
$$ f \circ \gamma - f = \Psi' (\tilde \varphi \gamma \tilde
 \varphi^{-1}) - \Psi(\gamma) \in \R\,.$$
Hence $\,{\rm Im}\, f\,$ is $\,\Gamma$-invariant. Then the analogous argument as in the 
claim in the proof of Proposition \ref{TYPECH} implies that $\,f \equiv r$, with $\,r \in\C$.
In particular $\, \Psi' (\tilde \varphi \gamma \tilde
 \varphi^{-1}) = \Psi(\gamma)\,$ and 
 $\,H$, $\,H'\,$ are either both of finite width or of infinite
width. Assume that, e.g.  $\,H= \{0 <  {\rm Im}\, z<C\}\,$  and 
$\,H'= \{0 <  {\rm Im}\, z<C'\}$. By applying $\,\tilde F\,$ to any $\,(z,w) \in H \times \Delta\,$
one sees that  $\ 0 < {\rm Im}\, z \  $ if and only if  $\, 0 <  {\rm Im}\, z + {\rm Im}\, r $.
 This implies that  $\,{\rm Im}\,r=0$, i.e. that $\,r\,$ is a real number and consequently $\,C=C'$.
 An analogous argument applies to the case when $\,H\,$ has infinite width. 
 \qed
\end{rem}

 \bigskip
 Let $\,X\,$ be a 2-dimensional, taut, Stein manifold with a proper 
$\,\R\,$-action. By Theorem \ref{BUNDLE}, the $\,\C$-action 
on $\,X^*\,$ is proper and one can consider the associated
holomorphic principal $\,\C$-bundle
$$\,\Pi: X^* \longrightarrow S :=X^*/\C\,.$$ 
If $\,S\,$ is compact, we show that $\,X\,$ is 
 $\,\R$-equivariantly biholomorphic to a
model of type CH. Then
Proposition \ref{TYPECH} implies that the globalization $\,X^*\,$ is Stein. We need
a preliminary result. Let the functions $\,\alpha$, $\,\beta\,$ be defined as in Lemma 
\ref{ALPHABETA}.


\bigskip
 \begin{lem}
 \label{COMPACT}
If $\,S\,$ is compact then 
 $\,\alpha$, respectively $\,\beta$, is either pluriharmonic or constantly  equal to
 $\, - \infty$.
 \end{lem}   
 
 \begin{proof}
Assume that 
 $\,\alpha\,$ is not constantly equal to $\,- \infty$. 
Since  by (ii) of Lemma \ref{ALPHABETA} one has 
$\,\alpha(\lambda \cdot x ) = - {\rm Im}\, (\lambda) + \alpha(x)\,$
for all $\,x \in X^*\,$ and $\,\lambda \in \C$,
 the real, positive $\,(1,1)$-current 
 $\,i\partial  \bar{\partial} \alpha\,$ is $\,\C$-invariant. Therefore it
 pushes down to  a  $\,(1,1)$-current 
 $ \,\theta\,$ on $\,S\,$ such that $\,\Pi^*(\theta) = i \partial  \bar{\partial} \alpha.$
 Note that  $\,\theta\,$ is also positive. 
  
Recall that  all cohomology groups with values in the sheaf of  smooth functions on $\,S\,$ vanish. Thus $\,X^*\,$ is 
 trivial as a differentiable principal $\,\C$-bundle and the
 maps induced by $\,\Pi\,$ in cohomology are isomorphisms.
 Since  $\, i\partial  \bar{\partial} \alpha\,$ is an exact current,  this implies
 that $\,\theta\,$ is also an exact current.
Then, from (ii) of Lemma \ref{POTENTIAL} it follows that   
$\,\theta = 0\,$ and consequently  
 $\, i\partial  \bar{\partial} \alpha=\Pi^*(\theta)=0$. Hence $\,\alpha\,$ is pluriharmonic.
 An analogous argument applies to show that if the function
  $\,\beta\,$ is not constantly equal to $\,- \infty$, then  it is  is pluriharmonic.  
 \end{proof}   


\bigskip
 \begin{pro}
 \label{ELIMINA}
  Let $\,X\,$ be a 2-dimensional, taut, Stein manifold with a proper 
$\,\R\,$-action and assume that $\,S :=X^*/\C\,$ is compact.
Then $\,S\,$  is hyperbolic
and 
$\,X\,$ is $\,\R$-equivariantly biholomorphic
 to a model of type {\rm CH}. In particular $\,X^*\,$ is Stein.
 \end{pro}   
 
 \begin{proof}
 First note that  $\,S\,$ can not be biholomorphic to the Riemann
sphere. Indeed  $\,H^1(\P^1(\C),{\mathcal O}) = 0$, thus if $\,S=\P^1(\C)\,$ then
 $\,X^* = \C \times \P^1(\C)$. Moreover
  the functions $\,\P^1(\C) \to \R \cup\{\infty\}$,
 defined by $\,p \to \alpha(0,p)\,$ and 
 $\,p \to \beta(0,p)$, are constant, being subharmonic on $\,\P^1(\C)$.
Since (ii) of Lemma \ref{ALPHABETA} implies that  
$\,X=\{\,(z,p) \in \C \times \P^1(\C) \ :\  \alpha(0,p)<{\rm Im}\,z< -\beta(0,p) \}$,
it follows that $\,X\,$ is  the product of a strip in $\,\C\,$ and $\,\P^1(\C)$.
However $\,X\,$ is Stein, therefore this is impossible.
 
Now let us show that $\,S\,$ is hyperbolic.
Consider the universal covering space  $\,\pi:\tilde X^* \rightarrow X^*\,$ of
$\,X^*\,$ with deck transformation 
group $\,\Gamma.$  The proper $\,\C$-action on $\,X^*\,$ lifts to a proper 
$\,\C$-action
on $\,\tilde X^*$, therefore  $\,\tilde X^*\,$ is a principal $\,\C$-bundle over
 $\,\tilde S \cong \tilde X^* / \C$.
Note that $\,X^*\,$ and $\,\tilde X^*\,$ are  trivial as differentiable principal $\,\C$-bundles over
$\,S\,$ and $\,\tilde S$, respectively. This implies that the Riemann surface
$\,\tilde S\,$ is simply connected, therefore it is non compact and consequently 
$\,\tilde X^*\,$ is $\,\C$-equivariantly
biholomorphic to $\,\C \times \tilde S\,$ (cf.\,Thm.\,\ref{BUNDLE}).
One has a commutative diagram
of holomorphic maps
$$\begin{matrix} 
  \tilde X^*= \C \times \tilde S &   \buildover{ \pi }  \to  &   X^*= \C \times \tilde S/ \Gamma \cr   
   \downarrow  &                       & \downarrow\ \ \ \  \ \ \ 
                                   &                 &                \cr
\tilde S       &       \buildover{ \hat \pi }   \to &   S = \tilde S/\Gamma\ \ \  \,,\cr 
\end{matrix} $$

\nsmallskip
where $\,\hat \pi\,$ is the universal covering of $\,S\,$ with deck transformation group $\,\Gamma$. 


By (iii) of Lemma \ref{ALPHABETA},
 the sum $\, \alpha + \beta\,$ is  $\,\C$-invariant, thus it can
be regarded as a subharmonic function on $\,S$.
Since  $\,S\,$ is compact, $\, \alpha + \beta\,$ is constant.
Moreover polar sets have zero measure, therefore 
if $\,\alpha + \beta \equiv - \infty\,$ then  either $\, \alpha \equiv - \infty\,$ or 
 $\,\beta \equiv - \infty.$
 As a consequence $\,X = \{ \,x \in X^*\  :\  \alpha(x) < 0 \,\}\,$
 or $\,X = \{ \, x \in X^* \ :\  \beta(x) < 0\, \}.$
 On the other hand,
 if $\,\alpha + \beta = -C\,$ for some positive real number $\,C$, 
 one has  $\,X = \{ x \in X^* :   \alpha(x) < 0< -\beta(x) \}\, =\,  \{ x \in X^* : - C < \alpha(x) < 0 \}$.

 First consider the case when $\,X = \{ \,x \in X^*\  :\  \alpha(x) < 0 \,\}$.
Set $\,\tilde \alpha = \alpha \circ \pi\,$ and let 
 $\,\tilde X :=\pi^{-1}(X) = \{ (z,p) \in \C \times \tilde S \ : \  \tilde \alpha(z,p) < 0    \}$.
 Recall that $\,\alpha\,$ is  pluriharmonic by Lemma  \ref{COMPACT},
 therefore so is $\,\tilde \alpha$.
Since $\,\tilde X^*\,$ is simply connected, (iii) of Lemma \ref{POTENTIAL}
implies that there exists a holomorphic function
$\, f : \tilde X^* \rightarrow \C\, $ such that  $\,{\rm Im} \,(f) = \tilde \alpha.$
Moreover, for all $\,(z,p) \in \tilde X^* \,$ one has 
$$\,\tilde \alpha(z,p )= \alpha \circ \pi(z \cdot (0,p) ) =\alpha(z \cdot \pi (0,p))  
=\tilde \alpha(0,p) - {\rm Im} \, z\,,$$ therefore 
 $ f(z,p) = f(0,p) - z$.  
 Then the map defined by 
 $$\, (z,p) \rightarrow (-f(z,p),p)=(z-f(0,p),p) \,$$
  gives a $\,\C$-equivariant biholomorphism 
of $\,\C \times \tilde S\,$ and its restriction to 
 $\,\tilde X\,$  defines an $\,\R$-equivariant biholomorphism onto  $\, \{ (z,p) \in \C \times \tilde S : 0 < {\rm Im}\, z \}$, which is simply connected. Thus $\,\tilde X\,$ 
 can be regarded as the universal covering of $\,X\,$ 
 and since  $\,\tilde X\,$ is taut  by Proposition \ref{COVERING},
    this implies that $\,\tilde S \cong \Delta$, i.e. that $\,S\,$ is hyperbolic.
   
    An analogous argument applies to the cases when 
    $\,X = \{ \, x \in X^* \ :\  \beta(x) < 0\, \}\,$ and
    $\,X = \{\, x \in X^* \ :\  - C < \alpha(x) < 0 \,\}$, showing that $\,S\,$
    is hyperbolic and that $\, \tilde X\,$ is $\,\R$-equivariantly biholomorphic to 
    $\,H \times \Delta$, where $\,H\,$ is given by $\, \{  0 < {\rm Im}\, z \}$,
    $\, \{  {\rm Im}\, z<0 \}\,$ or
 $\, \{ 0< {\rm Im}\,z<C \}$,  for some positive real $\,C$.    
    
    Identify the universal covering $\,\tilde X\,$ with  $\,H \times \Delta\,$ and note that it is 
    $\,\Gamma$-invariant in $\,X^* \cong  \C \times \Delta$.
    In order to describe the  $\,\Gamma$-action, observe that
   every $\,\gamma\,$ in $\,\Gamma\,$ is $\,\C$-equivariant, therefore
there exists a holomorphic map 
 $\,F_\gamma : \Delta \to \C\,$ such that 
 $$\gamma \cdot (z,w) = (z + F_\gamma(w),\gamma \cdot w)\,,$$
 for all  $\,(z,w) \in \C \times \Delta$.
 Since $\,\gamma(H \times \Delta)=H \times \Delta$, it follows that 
$\,{\rm Im}\, F_\gamma \equiv 0\,$ and consequently the holomorphic function
$\,F_\gamma\,$ is a real constant.
Thus
the  $\,\Gamma$-action on  $\,\C \times \Delta\,$ is given 
by $\,\gamma \cdot (z, w)= (z + \Psi(\gamma), \gamma \cdot w)$,
where the group homomorphism $\,\Psi: \Gamma \to \R\,$
is defined by $\,\gamma \to F_\gamma$. 

Finally note that $\, \Ker\,\Psi \not = \Gamma$. Otherwise one has 
$\,X= H \times \Delta/\Gamma =  H \times( \Delta/\Gamma ) =H \times S$.
Since $\,X\,$ is Stein and $\,S\,$
is compact, this gives a contradiction. Thus $\,X\,$
is $\,\R$-equivariantly biholomorphic to a model of type CH and
$\,X^*\,$ is Stein by (i) of Proposition
\ref{TYPECH}.
\end{proof}

 \bigskip
 \begin{rem}
 \label{NONEQ}
Note that two models of type CH, one  of the form 
$\,H \times \Delta/\Gamma$, with $\,H\,$ of finite width,
and one of  the form  $\,H' \times \Delta/\Gamma'$,
with $\,H'\,$ of infinite width, cannot be biholomorphic.
Let  $\,\Gamma\,$ act on $\,H \times \Delta\,$
by   $\,\gamma \cdot (z,w) = (z+ \Psi(\gamma), \gamma \cdot w)\,$  and 
let $\,\Gamma'\,$ act on $\,H' \times \Delta\,$
by  $\,\gamma' \cdot (z,w) = (z+ \Psi'(\gamma'), \gamma' \cdot w)$,
where $\,\Psi:\Gamma \to \R\,$ and $\,\Psi': \Gamma' \to \R\,$
are non trivial homomorphisms. 
Recall that the elements of $\,\Gamma\,$ and $\,\Gamma'\,$ 
are all hyperbolic, i.e. they have two fixed point on the boundary of $\,\Delta$
(see \cite{FaKr}, Cor.$\,$2,$\,$p.$\,$216).

In particular every element of $\,\Gamma\,$ which does not belong to $\,\Ker\, \Psi\,$
has 4 fixed points on the boundary of the universal covering $\,H \times \Delta\,$
of $\,X$, while  an element 
of $\,\Gamma'\,$ has either infinite or 2 fixed points 
on the boundary of
$\,H' \times \Delta$.
\qed
\end{rem}

\bbigskip


\section{models with non compact base}

  \bigskip

Here we consider the models with
base a non compact Riemann surface. Let us start with the hyperbolic case.

\nbigskip
{\bf Type NCH} $\ $Let $\,S\,$ be a non compact hyperbolic Riemann surface.
A model of type NCH with base $\,S\,$ is given by
 $$\, \{\, (z,p) \in \C \times S : a(p) < {\rm Im}\, \,z < - b(p) \,  \}\,,$$ where 
$\,a\,$ and $\, b\,$ are  subharmonic,  continuous functions on $\,S\,$
such that $\, a +b < 0\,$ and  $\, \max \{a(p),b(p)\}  >   - \infty\, $  for all $\,p \in S$.

\nbigskip
{\bf Type NCNH} A model of type NCNH with base $\,S = \C\, $ or $\,S= \C^*\,$
is given by  
$$\,\{\, (z,p) \in \C \times S \ :\  a(p) < {\rm Im} \,z \, \}\quad {\rm or} \quad 
\{\, (z,p) \in \C \times S \ : \ {\rm Im} \,z < -b(p) \, \}\, ,$$
with $\,a$, $\,b\,$   subharmonic, non-harmonic,  real valued,  continuous functions
 on $\,S$.

\nbigskip
On each manifold as above let $\,\R\,$ act  by translations on the first factor.

\bigskip
\begin{pro}
\label{TYPENCH}
Let $\,X\,$ be a model of type {\rm NCH}  or {\rm NCNH} with base $\,S$. Then

\medskip
\item{\rm (i)} the universal globalization of $\,X\,$ is $\,\C \times S$, which is Stein,
\medskip
\item{\rm (ii)}  $\,X\,$ is a taut, Stein manifold with a proper $\,\R$-action.

\end{pro}

\begin{proof}
(i) Note that $\,X\,$ is orbit-connected in $\,\C \times S$. Then
Lemma$\,$1.5 in \cite{CIT} implies that $\,\C \times S\,$
is the (Stein) universal globalization of $\,X$. 

(ii) Since $\,X\,$ is 
an $\,\R$-invariant submanifold in $\,\C \times S$,
the $\,\R$-action on $\,X\,$
is proper. 
Moreover, $\,X\,$ is given as
the sublevel set of plurisubharmonic functions defined
on the product $\,\C \times S$, which is Stein. Thus $\,X\,$ is Stein.

Finally we show that $\,X\,$ is taut. For $\,X\,$ a model 
of type NCH
consider the projection 
$\,\Pi |_X: X \to S\,$, $\,(z,p) \to p$, onto
 the second factor.
  By Lemma \ref{PREIM} it is sufficient to prove that for every $\,p\,$ in 
  $\,S\,$ there exists a neighborhood $\,U\,$ of $\,\,p$ in $\,S\,$ 
  such that  $(\Pi |_X)^{-1}(U)$ is taut.
Since   $\,\max\{a(p),b(p)\}> - \infty\,$ we may
assume that, e.g. $\,a(p) > - \infty.$ 
By continuity  $\,a > M\,$ on a neighbourhood $\,U\,$ of $\,p\,$, for some
real constant $\,M$. Then  $\,(\Pi |_X)^{-1}(U)\,$ is contained in
$\,H \times S$, with $\, H=\{ z\in \C \ : \   M< {\rm Im}\, z  \}$. Moreover
 the inverse image $\,(\Pi |_X)^{-1}(U)\,$
it is defined as a sublevel
set of continuous plurisubharmonic functions, therefore 
it is taut by Corollary \ref{SUBLEVEL}.
 
Assume now that $\,X\,$ is a model of type NCNH.
Note that if $\,S=\C^*\,$, then the universal covering $\,\tilde X\,$ of $\,X\,$ is contained in $\,\C^2\,$ and it is also 
of type NCNH.  Moreover, by Proposition \ref{COVERING} the manifold
$\,X\,$ is taut if and only if so is $\,\tilde X$.
Thus
we may  assume that $\,S=\C$. Since a domain of the 
form $\,\{\, (z,w) \in \C^2 \ : \ {\rm Im} \,z < -b(w) \, \}\, $ 
is biholomorphic to $\,\Omega_{b}\,$ via the biholomorphism of $\,\C^2\,$ defined
by $\,(z,w) \to (-z,w)$,  all such models are taut by
Theorem \ref{TAUTT}. 
 \end{proof}

\bigskip
\begin{rem}
 \label{BIHOLOCHNC}
Let $\,F\,$ be a $\,\R$-equivariant biholomorphism between 
two models of type NCH  defined by 
$\,\{ \,(z,p) \in \C \times S \ :\  a(p) < {\rm Im}\, z  < -b(p)\,\}\,$ and 
$\, \{\, (z,p') \in \C \times S'\  : \ a'(p') < {\rm Im}\, z<-b'(p')  \,  \}\,.$
Then 
$ F(z,p) = (z + f(p),\varphi(p))\,$ where 
$\,f: S \rightarrow \C\,$  is holomorphic and  
$\,\varphi\ : S \rightarrow S'\,$
is a biholomorphism  such that
$\, a = (a' \circ \varphi - {\rm Im}\, f)\,$ and $\,b = (b' \circ \varphi + {\rm Im}\, f)$.
An analogous statement holds for models of type NCNH.
 \qed
 \end{rem}

 \bigskip
 \begin{pro}
 \label{CORNCH}
Let $\,X\,$ be a 2-dimensional, taut, Stein manifold with a proper  
$\,\R$-action and assume that  the Riemann surface 
$\,S:=X^*/\C\,$ is non compact.
Then $\,X^*\,$ is $\,\C$-equivariantly biholomorphic to $\,\C \times S$, which
is Stein. 
Moreover, depending on hyperbolicity of $\,S\,$ the manifold 
$\,X\,$ is $\,\R$-equivariantly biholomorphic either 
 to a model of type {\rm NCH} or {\rm NCNH}.
\end{pro}

\begin{proof}
Since $\,S\,$ is non compact by assumption, the principal $\,\C$-bundle 
$\,X^*$ is trivial (cf.\,Thm.\,\ref{BUNDLE}), implying the first statement.
Regard  
$\,X\,$ as $\,\{\,(z,p) \,\in \C \times S \ :  \ \alpha(z,p) < 0 < -\beta(z,p)\,\}\,$
and define $\,a(p):= \alpha(0,p)\,$ and $\,b(p):=\beta(0,p)$.
Since from (ii) of  Lemma \ref{ALPHABETA} it follows that $\,\alpha(z,p) = -{\rm Im}\,z + \alpha(0,p)\,$ and 
$\,\beta(z,p) =  {\rm Im} \,z + \beta(0,p)$, one has 
$$\,X=\{\,(z,p) \,\in \C \times S \ :  \ a(p) <  {\rm Im}\,z < -b(p)\,\}\,.$$
Moreover the same lemma implies that 
$\,a\,$ and $\, b\,$ are  subharmonic,  continuous functions,
$\, a +b < 0\,$ and  $\, \max \{a(p),b(p)\}  >   - \infty\, $  for all $\,p \in S$.
This concludes the case when $\,S\,$ is hyperbolic.

For $\,S = \C $ or $\,S= \C^*\,$ we first note that $\,a + b\,$ is constant,
being a subharmonic, negative function on $\,S$. We claim that 
$\,a + b \equiv - \infty$.  Assume by contradiction that 
$ a+b = -C $ for some positive  $\,C\,$. Then $\,a= -b-C\,$ is
harmonic and  $\,X = \{\, (z,p) \in \C \times S : a(p) < {\rm Im}\, z <a(p) + C \,\} .$ 
In the case when $\,S = \C,$  there exists a holomorphic  function 
$\,f: \C   \rightarrow \C$ such that $\,{\rm Im}\, f=a$. Then the map 
$\, \zeta \rightarrow(f(\zeta) + iC/2, \zeta)\,$ is  a non constant holomorphic
map from $\,\C\,$ into $\,X$. Since $\,X\,$ is taut, this gives a contradiction.
If  $\,S = \C^*$, one can show that $\,a+b \equiv -\infty\,$ by 
applying  the analogous argument to 
the universal covering of $\,X$, which is taut by Proposition \ref{COVERING}.

Thus $\,a + b \equiv - \infty\,$ and since
the sets $\,\{a=-\infty\}\,$ and $\,\{b=-\infty\}\,$ have zero measure, either
$\,a\,$ or $\,b\,$ are constantly equal to $-\infty$. Assume, e.g. that 
$\,b \equiv -\infty$.
Since $\,X\,$ is taut, $\,a\,$ is necessarily real valued and 
the above argument also proves that $\,a\,$ can not be harmonic.
Thus $\,X\,$ is $\,\R$-equivariantly biholomorphic to a model of type
NCNH.
 \end{proof}

\bigskip
\begin{cor}
\label{AREALL}
Let $\,S\,$ be a non compact Riemann surface and consider the subdomain 
of $\,\C \times S\,$ defined by 
$$\, \Omega :=\{\, (z,p) \in \C \times S : a(p) < {\rm Im}\, \,z < - b(p) \,  \}\,,$$
where 
$\,a, b : S \to \{-\infty\} \cup \R\,$ are upper semicontinuous functions.
Then $\,\Omega\,$ is taut if and only if it is a model of type {\rm NCH} or {\rm NCNH}.
\end{cor}

\begin{proof}
First note that if $\,\Omega\,$ is taut, then it is Stein. For this consider
the universal covering 
$\,Id \times \pi: \C \times \tilde S \to \C \times S$, where $\,\tilde S=\C\,$ or
$\,\tilde S=\Delta\,$. Then Proposition \ref{COVERING} applies to show that 
the inverse image $\,(Id \times \pi)^{-1}(\Omega)\,$ is a taut
domain of $\,\C^2$.
Thus  it is Stein by Thm.$\,$5.4.1 in \cite{Kob}
 and consequently  it is locally Stein in $\,\C \times \tilde S$.
It follows that  
$\,\Omega\,$ is locally Stein in $\,\C \times S$, which  is Stein. Thus 
$\,\Omega\,$ is Stein by \cite{DoGr}. 

Then an analogous argument as in the above proof applies to prove that
$\,\Omega\,$ 
is a model of type {\rm NCH} or {\rm NCNH}.
\end{proof}

\nbbigskip


\section{Homotopy of the models}

\bigskip
Let us summarize the main results of the previous sections as follows 
(see Prop. \ref{TYPECH}, \ref{ELIMINA}, \ref{TYPENCH} and \ref{CORNCH}).

\bigskip
\begin{theorem}
\label{MAIN}
Every model of type {\rm CH}, {\rm NCH} or {\rm NCNH} is taut and Stein.
Moreover a 2-dimensional, taut, Stein manifold with a proper 
$\,\R$-action is $\,\R$-equivariantly
biholomorphic to one of them. In particular 
its universal globalization is Stein.
\end{theorem}

\nbigskip

Here we show that in most cases, but not all of them,
the type of a 2-dimensional, taut, Stein manifold
with a proper  $\,\R$-action  is uniquely
determined by its topology.


\bigskip

\begin{pro}
\label{HOMOT}
 Let $\,X\,$ be a 2-dimensional, taut, Stein manifold with a proper  $\,\R$-action
Then $\,X\,$ is homotopically 
equivalent to $\,S$.
\end{pro} 

\begin{proof}
We first find a smooth global section of the restriction
 $\,\Pi|_X : X \rightarrow S\,$ of $\,\Pi\,$ to $\,X$.  
In a given smooth, trivialization $\,\C \times S\, $ of $\,X^*\,$ one has 
$X = \{ (z,p)\in \C \times S\  :\  a(p) < {\rm Im}\, z < - b(p) \}$, with $\,a\,$ and $\,b\,$
continuous functions on $\,S\,$ (maybe no longer subharmonic)
with values in $\,\{- \infty \} \cup  \R$. 
Moreover, by continuity of $\,a\,$ and $\,b\,$
one can choose  a locally finite covering  $\,\{ U_j \}\,$
of $\,S\,$ and  real constants $\,M_j\,$ such that 
 $\,a<M_j<-b\,$ on each $\,U_j$. Thus the constant functions
  $\,iM_j\,$ can be regarded as 
 smooth local sections of  $\,\Pi |_X$. Choose a 
 smooth partition of unity  $\,\{\psi_j\}\,$
 subordinated to $\,\{ U_j \}$. Then $\,\theta = \sum_j i M_j \psi_j \,$
 defines  a smooth global section of  $\,\Pi|_X$, since $\,  a(p) < {\rm Im}\, \theta < - b(p) \,$
 on $\,S$.
 
 Finally note that the map $\,X\times [0,1] \to X\,$ defined
 by
 $\,((z,p),t) \to (z + t(\theta(p)-z), \,p)\,$ is a  
homotopy equivalence, showing that $\,S\,$ is a strong deformation retract of 
 $\,X$.
  \end{proof}

 \bigskip
 \begin{cor}
 \label{TOPOL}
  Let $\,X\,$ be a 2-dimensional, taut Stein manifold with a proper  $\,\R$-action.
Then 

\medskip
\item{\rm (i)}
$\,X\,$ is $\,\R$-equivariantly biholomorphic to a model of type {\rm CH} 
if and only if $\,H^2(X,\Z) \neq 0\,$, 

\medskip
\item{\rm (ii)}
 if $\,H^2(X,\Z) = 0\,$
and $\,\pi_1(X)  \, $  is neither trivial,  nor isomorphic to  
 $\,\Z$,  then $\,X\,$ is $\,\R$-equivariantly biholomorphic to a 
 model of type {\rm NCH},

\medskip
\item{\rm (iii)}
if $\pi_1(X) = 0 \ $ or  
 $\,\pi_1(X) = \Z,$ then $\,X\,$ is $\,\R$-equivariantly biholomorphic to 
 a model of type {\rm NCH} {\rm (}type {\rm NCNH}{\rm )} if and only if it
 {\rm (}does not{\rm )} 
 admits a non constant, bounded
 holomorphic $\,\R$-invariant function.
  \end{cor}

\begin{proof}
(i) and (ii) are direct consequences of the above proposition.
For (iii)  note that an $\,\R$-invariant holomorphic function
on $\,X\,$ pushes down to a holomorphic function on $\,S$. 
Moreover, the assumption
on the fundamental group implies that $\,S\,$ is biholomorphic
to one of the following domains $\,\C$, $\,\C^*$, $\,\Delta$, $\,\Delta^*\,$ or an annulus.
This implies the statement.
 \end{proof}

 \bigskip

 \begin{rem}
 \label{ESEMPI} Let $\,X\,$ be a taut, Stein surface such that 
 either $\,\pi_1(X) = 0\,$ or $\,\pi_1(X) = \Z$. Then, for different proper 
 $\,\R$-actions, the manifold $\,X\,$ may be $\,\R$-equivariantly biholomorphic
 to models of
 different types.
 
 As an example consider the unbounded realization of the unit ball
 of $\,\C^2\,$ given by $\,X=\{\,(u,v) \in \C^2 \ : |v|^2<{\rm Im}\,u \}\,$ and the two
 different $\,\R$-actions on $\,X\,$ defined by 
  $$t\diamond (u,v):= (u +t,v), \quad \quad t\ast (u,v):= (u-2tv+it^2 , v-it)\,.$$
Such actions appear in \cite{FaIa} as normal forms of parabolic elements
in the automorphism group of $\,X$.
 It is clear that the globalization with respect to the first action is
 $\,\C^2$ and its $\,\C$-quotient is $\C$.
 
 Note that the second $\,\R$-action extends to a $\,\C$-action on $\,\C^2\,$
 and a simple computation shows that
 $\,\C \ast X = \{\,(u,v) \in \C^2 \ : {\rm Im}\, u > ({\rm Im}\, v)^2 - ({\rm Re\, v})^2\, \}$.
 Moreover, one checks that $\,X\,$ is orbit-connected in  $\,\C \ast X$. Then
Lemma$\,$1.5 in \cite{CIT} implies that $\,\C \ast X\,$ is the universal globalization
 with respect to the local $\,\C$-action on $\,X\,$ induced by the second $\,\R$-action. 
 Let $\,\H=\{ \,z \in \C \ : \ 0<{\rm Im}z \}$.
 One has a $\,\C$-equivariant  biholomorphism
 $$\,\Psi: \C \times \H \to  \C \ast X\,, \quad \quad 
 (\lambda, u) \to \lambda \ast (u,0) = (u + i\lambda^2,-i \lambda)\,.$$ 
 Therefore  the $\,\C$-quotient of  $\,\C \ast X\,$ is biholomorphic 
 to $\,\H$.
 This completes the example in the case when the fundamental group
  is trivial. 
 
 A similar example with fundamental group isomorphic to $\,\Z\,$
 is as follows. Since $\,\Psi\,$ is a biholomorphism,
 $\,X':=\Psi^{-1}(X)=\{\,(u,v) \in \C^2 \ : \  {\rm Im}\, v > 2 ({\rm Im}\, u)^2 \}\,$ is also a model for the unit ball of $\,\C^2$. On this model the above actions look like
 $$t\diamond (u,v):= (u,v+t), \quad \quad t\ast (u,v):= (u+t, v)\,.$$
 Then the $\,\Z$-action on $\,X'\,$ defined by 
 $\,n\cdot (u,v):= (u+n, v+n)\,$ commutes with both the $\,\R$-actions.
 Thus such $\,\R$-actions push down to proper $\,\R$-actions on the quotient 
 $\,X'/\Z$, whose fundamental group is $\,\Z$.  Observe that 
  $\,X'/\Z\,$ is taut by Prop. \ref{COVERING} and it is Stein by 
  \cite{FaIa}
 
 Finally note that the restrictions $\,X' \to \C\,$ and $\,X' \to \H\,$
 to $\,X'\,$ of the projections of the associated holomorphic principal
 $\,\C$-bundles are $\,\Z\,$ equivariant. Thus they factorize to
 the restrictions $\,X'/\Z\to \C^*\,$ and $\,\:X'/\Z \to \Delta^*=\H/\Z\,$
 to $\,X'/\Z\,$ of the projections of the holomorphic principal
 $\,\C$-bundles associated to the pushed down $\,\R$-actions on $\,X'/\Z$.
 Since the bases 
 of this bundles are $\,\C^*\,$ and $\,\Delta^*$, this shows that also in this case
 the type of $\,X\,$ depends on the chosen $\,\R$-action.
  \qed
  \end{rem}

  \hfill

\nbbigskip

\section{ taut Hartogs  domains}

\bigskip

As an application of the given classification,
we give necessary and sufficient conditions for 
tautness of (non-complete)  Hartogs domains over
a non compact  Riemann surface $\,S$.
A complete Hartogs domain over $\,S\,$ is given by
$$\,\{\, (u,p) \in \C \times S \ : \  |u| < e^{-b(p)}\, \}\,,$$ 
with  $\,b : S \rightarrow \R \cup \{ - \infty \}\, $ an 
upper semicontinuous function. 
A non-complete Hartogs domain over $\,S\,$ is given  by
$$\,\{\, (u,p) \in \C \times S \ : \  e^{a(p)}< |u| < e^{-b(p)}\, \}\,,$$
where $\,a,b : S \rightarrow \R \cup  \{- \infty \} \,$ 
are upper semicontinuous functions with $\,a+b<0$.

We wish to determine under which conditions on $\,a\,$ and $\,b\,$ such
domains  are taut. 
A result of Thai-Duc (\cite{ThDu}), which applies in a more general context, implies that
a complete Hartogs domain  is taut if and only if $\,S\,$ is hyperbolic and 
 $\,b\,$ is a real valued, subharmonic, continuous
function.
 The following proposition gives a characterization of non-complete Hartogs domains
  (cf.\,\cite{Par} for related results).

\bigskip
\begin{pro}
\label{HARTOGS}
Let $\, \Omega = \{ (u,p) \in \C \times S \ : \ e^{a(p)} <  |u| < e^{-b(p)} \} $ be a  non-complete Hartogs domain over a non compact Riemann surface.
Then $\,\Omega\,$ is taut if and only if either 

\medskip
\item{\rm (i)} 
the Riemann surface $\,S\,$ is hyperbolic, the functions $\,a$, $\,b\,$ are  continuos subharmonic  and
$\,max(a(p),b(p))> -\infty\,$ for every $\,p\in S$, or

\medskip
\item{\rm (ii)} 
the Riemann surface $\,S\,$ is not  hyperbolic, $\,b \equiv - \infty\, $
{\rm(}respectively  $\,a \equiv - \infty\,${\rm)}
and $\,a\,$ {\rm(}respectively  $\,b${\rm)}
is a subharmonic, non-harmonic, real-valued, continuous
function.

\end{pro}

\begin{proof}
Consider the covering map $\,F: \C \times S \to \C^* \times S\,$ given  by
$\,(z,p) \to (e^{-iz},p)$.  Since the restriction of $\,F\,$ to  $\, F^{-1}(\Omega)\,$ is a covering,
from proposition \ref{COVERING} it follows that 
 $\, F^{-1}(\Omega)\,$ is taut if and only if so is $\,\Omega$. 
Note that $\,\Omega\,$ 
 is invariant under the $\,S^1$-action defined by  $\,e^{i \theta} \cdot (u,p) := (e^{i \theta} u,p).$ As a consequence $\,\R\,$ acts properly on  $\, F^{-1}(\Omega)\,$  
by $\,t \cdot (z,p):= (z+t,p)$.
Then Corollary \ref{AREALL} applies to show that $\, F^{-1}(\Omega)\,$ is
taut if and only if it is 
a model of type NCH or NCNH, depending on hyperbolicity of $\,S$. This implies the
statement.
\end{proof}

 \hfill
  

\end{document}